\documentclass[leqno]{article}
\usepackage{amsmath}
\usepackage{amssymb}
\usepackage{amsfonts}
\usepackage{mathbbol}
\usepackage{amsthm}
\usepackage[dvips]{graphicx}
\usepackage[all]{xy}

\makeatletter
 \@addtoreset{equation}{section}
 
 \makeatother

\theoremstyle{plain}
\newtheorem{thm}{Theorem}[section]

\newtheorem{lem}[thm]{Lemma}
\newtheorem{cor}[thm]{Corollary}

\title{Critical probability on the product graph \\
		of a regular tree and a line}
\author{Kohei Yamamoto}
\date{Tohoku university}
\begin{document}
\maketitle

\begin{abstract}
	We consider Bernoulli bond percolation on the product graph of
	a regular tree and a line.
	Schonmann showed that there are a.s. infinitely many infinite clusters at $p=p_u$
	by using a certain function $\alpha(p)$.
	The function $\alpha(p)$ is defined by a exponential decay rate of probability that
	two vertices of the same layer are connected. 
	We show the critical probability $p_c$ can be written by using $\alpha(p)$.
	In other words, we construct another definition of the critical probability.
\end{abstract}

\section{Introduction}
\label{se:Intro}
	Let $G=(V,E)$ be a connected, locally finite and infinite graph,
	where $V$ is the set of vertices, $E$ is the set of edges.
	In Bernoulli bond percolation, 
	each edge will be open with probability $p$,
	and closed with probability $1-p$ independently, 
	where $p \in [0,1]$ is a fixed parameter.
	Let $\Omega=\{0,1\}^E$ be the set of samples,
	where $\omega(e)=1$ means $e$ is open.
	Each $\omega \in \Omega$ is regarded as a subgraph of $G$ 
	consisting of all open edges.
	The connected components of $\omega$ are referred to as clusters.
	Let $p_c=p_c(G)$ be the critical probability for Bernoulli bond percolation 
	on $G$, that is,
	\[
	p_c= \inf \left\{ p\in[0,1] \mid 
	\text{there exists an infinite cluster almost surely} \right\},
	\]
	and let $p_u=p_u(G)$ be the uniqueness threshold for Bernoulli bond percolation 
	on $G$, that is,
	\[
	p_u= \inf \left\{ p\in[0,1] \mid 
	\text{there exists a unique infinite cluster almost surely} \right\}.
	\]
	One of the most popular graphs in the theory of percolation is 
	the Euclidean lattice $\mathbb{Z}^d$.
	In 1980 Kesten \cite{Kesten} proved that 
	$p_c=1/2$ in the case of two dimensions.
	But in the case of three dimensions or more,
	as a numerical value, the critical probability is not quite clear. 
	Regarding the uniqueness threshold of the Euclidean lattice,
	in 1987 Aizenman, Kesten, and Newman \cite{Aizenman} proved that there exists
	at most one infinite cluster almost surely for all $d \ge 1$,
	that is, they showed that  $p_c=p_u$ for all $d \ge 1$. 
	The product graph of a $d$-regular tree and a line $T_d \Box \mathbb{Z}$
	was presented as a first example of a graph with $p_c < p_u<1$
	by Grimmett and Newman \cite{Grimmett2} in 1990,
	where a product graph means a Cartesian product graph. 
	They showed that $p_c<p_u$ holds when $d$ is sufficiently large.
	After this article had appeared,
	percolation on $T_d \Box \mathbb{Z}$ has become a popular topic.
	However, the critical probability of $T_d \Box \mathbb{Z}$ 
	is, as a value, also not quite clear.
	In this paper we study Bernoulli bond percolation on $T_d \Box \mathbb{Z}$.
	Our goal is to write the critical probability by using a certain function $\alpha(p)$.
	From our theorem, we can consider a numerical value of the critical probability
	by analyzing $\alpha(p)$.
	We denote the probability measure associated 
	with Bernoulli percolation process by $\mathbb{P}_p$ or $\mathbb{P}_p^G$.
	Let $(x \leftrightarrow y)$ be an event 
	that there exists an open path between $x$ and $y$
	for two vertices $x,y \in V$.
	The function $\alpha(p)$ which was appeard in \cite{Schonmann} is defined by
	\[
	\alpha(p)=\alpha_d(p)
	=\lim_{n \to \infty} \mathbb{P}_p(o \leftrightarrow (v_n,0))^{\frac{1}{n}},
	\]
	where $v_n$ is a vertex on $T_d$ with $n$ distance from the origin.
	From a homogeneity of $T_d$, 
	$\alpha(p)$ does not depend on a choice of $v_n$.
	We abbreviate $v_n$ as $n$.
	We check on the existence of a limit.
	From FKG inequality, we have
	\[
	\mathbb{P}_p(o \leftrightarrow (n+l,0))
	\ge \mathbb{P}_p(o \leftrightarrow (n,0)) \mathbb{P}_p(o \leftrightarrow (l,0))
	\]
	for all $n,l \ge 0$.
	By using Fekete's subadditive lemma, the existence of the limit is ensured,
	and we have
	\[
	\alpha(p)
	=\lim_{n \to \infty} \mathbb{P}_p(o \leftrightarrow (n,0))^{\frac{1}{n}}
	=\sup_{n \ge 1} \mathbb{P}_p(o \leftrightarrow (n,0))^{\frac{1}{n}}.
	\]
	Let $B(k)$ be a $k$-ball of $T_d \Box \mathbb{Z}$ whose center is $o$, we have
	\[
	\alpha(p)
	= \sup_{n \ge 1} \sup_{k \ge 1} 
	\mathbb{P}_p^{B(k)}(o \leftrightarrow (n,0))^{\frac{1}{n}},
	\]
	and observe that we are taking the supremum of a continuous function of $p$.
	Therefore $\alpha(p)$ is lower semi-continuous 
	and, since it is clearly non-decreasing,
	it is also left-continuous.
	This function $\alpha(p)$ was first defined by Schonmann \cite{Schonmann}.
	Schonmann showed the following theorem.

	\begin{thm}[\cite{Schonmann}]
	\label{thm:Schonmann}
	Let $p_u$ be the uniqueness threshold. Then we have
	\[
	\alpha(p_u) \le \frac{1}{\sqrt{d-1}}.
	\]
	\end{thm}
	Schonmann considered percolation at critical point, $p=p_u$.
	Then Schonmann showed there are a.s. infinitely many infinite clusters
	by using this theorem.
	In percolation at another critical point, 
	$p=p_c$, we consider the value $\alpha(p_c)$.
	Hutchcroft showed the following theorem.
	
	\begin{thm}[\cite{Hutchcroft1}]
	Let $G$ be a quasi-transitive graph with exponential growth. Then
	\[
	\kappa_{p_c}(n)
	=\inf \left\{ \tau_{p_c}(x,y) \mid x,y \in V, d(x,y) \le n \right\}
	\le \mathrm{gr}(G)^{-n}
	\]
	for all $n \ge 1$,
	where $\tau_p(x,y)=\mathbb{P}_p(x \leftrightarrow y)$ and 
	$\displaystyle\mathrm{gr}(G)=\liminf_{r \to \infty} |B(x,r)|^{1/r}.$
	\end{thm}	
	The following lemma can be showed by using a similar arugument of this theorem.
	\begin{lem}
	\label{lem:alpha(p_c)}
	Let $G=T_d \Box \mathbb{Z}$. Then we have
	\[
	\alpha(p_c) \le \mathrm{gr}(G)^{-1}=\frac{1}{d-1}.
	\]
	\end{lem}
	
	Hutchcroft showed the following therorem.
	\begin{thm}[\cite{Hutchcroft2}]
	Let $G=T_d \Box \mathbb{Z}$.
	Then $p_c<p_u$ holds for all $d\ge3$.
	\end{thm}
	Then we can consider the value $\alpha(p)$ for $p \in (p_c, p_u)$.
	\begin{lem}
	\label{lem:alpha(p>p_c)}
	For all $p \in (p_c, p_u)$, we have 
	\[
	\alpha(p) \ge \frac{1}{d-1}.
	\]
	\end{lem}
	We consider other characteristics of $\alpha(p)$.
	
	\begin{thm}
	\label{thm:conti}
	The function $\alpha(p)$ is a strictly increasing on $[0,p_u]$,
	and a continuous on $[0,p_c]$.
	\end{thm}
	Especially, $\alpha(p)$ is continuous at $p_c$.
	Hence by using Lemma \ref{lem:alpha(p_c)} and Lemma \ref{lem:alpha(p>p_c)},
	we have the following corollary.
	
	\begin{cor}
	Let $p_c$ be the critical probability, Then we have
	\[
	\alpha(p_c)=\frac{1}{d-1}.
	\]
	\end{cor}

	The function $\alpha(p)$ is strictly increasing on $[0,p_u]$.
	Then we can define the inverse function of $\alpha$
	from $[0,\alpha(p_c)]$ to $[0,p_c]$. 
	By using this inverse function, we have the following theorem.
	\begin{thm}
	\label{thm:alpha.inverse}
	Let $G=T_d \Box \mathbb{Z}$. Then we have
	\[
	p_c=\alpha^{-1}\left(\frac{1}{d-1}\right).
	\]
	\end{thm}
	We must show Lemma \ref{lem:alpha(p_c)}, Lemma \ref{lem:alpha(p>p_c)}
	and Theorem \ref{thm:conti} to gain Theorem \ref{thm:alpha.inverse}.
	We will show  Lemma \ref{lem:alpha(p_c)} and Lemma \ref{lem:alpha(p>p_c)} 
	in Section \ref{sc:pf.lem}.
	In Section \ref{sc:extension} and Section \ref{sc:conn.event},
	we prepare some tools to show Theorem \ref{thm:conti}.
	and it will show in Section \ref{sc:beta} and Section{sc:pf.beta<1}.
	
	\section{Proof of Lemma \ref{lem:alpha(p_c)} and Lemma \ref{lem:alpha(p>p_c)}}
	\label{sc:pf.lem}
	We will require the following well-known theorem.
	\begin{thm}[\cite{Aizenman2}, \cite{Antunovic}]
	\label{thm:Aizenman2}
	Let $G$ be a quasi-transitive graph, and $o$ be a fixed vertex of $G$.
	Then we have
	\[
	\sum_{x \in V} \tau_p(o,x) <\infty
	\]
	for all $p<p_c$.
	\end{thm}
	This theorem was proven in the transitive case 
	by Aizenman and Barsky \cite{Aizenman2}, 
	and in the quasi-transitive case by Antunovi\'c and Veseli\'c \cite{Antunovic}.

	\noindent \textit{Proof of Lemma \ref{lem:alpha(p_c)}}. 
	Let $S(n)$ be a set of vertices of $T_d$ with $n$ distance from the origin.
	For all $p \in [0,1]$ and all $n \ge 1$, we have
	\[
	\tau_p(o, (n,0)) \cdot |S(n)|
	= \sum_{x \in T_d, |x|=n} \tau_p(o,(x,0))
	\le \sum_{x \in T_d \Box \mathbb{Z}} \tau_p(o,x).
	\]
	By using Theorem \ref{thm:Aizenman2},
	the right-hand side is finite when $p<p_c$.
	We know $|S(n)|=d(d-1)^{n-1}$.
	Then we have
	\[
	\lim_{n \to \infty} \tau_p(o, (n,0))^{\frac{1}{n}}
	\le \lim_{n \to \infty} \left( \frac{\sum \tau_p(o,x)}
	{|S(n)|} \right)^{\frac{1}{n}}
	=\frac{1}{d-1}.
	\]
	This means that $\alpha(p) \le 1/(d-1)$ for all $p<p_c$.
	Since $\alpha(p)$ is a left-continuous, we have $\alpha(p_c) \le 1/(d-1)$.
	This ends the proof of Lemma \ref{lem:alpha(p_c)}.
	\begin{flushright} $\Box$ \end{flushright}

	We will prepare some tools to show Lemma \ref{lem:alpha(p>p_c)}.
	Let $G_{\bullet,m}=T_d \Box [-m, m]$ for each $m \ge 0$.
	Similarly to $\alpha(p)$, we define $\alpha_m(p)$ by
	\[
	\alpha_m(p)=\lim_{n \to \infty} \mathbb{P}_p^{G_{\bullet,m}}
	(o \leftrightarrow (n,0))^{\frac{1}{n}}
	=\sup_{n\ge 1} \mathbb{P}_p^{G_{\bullet,m}}
	(o \leftrightarrow (n,0))^{\frac{1}{n}}.
	\]
	\begin{lem}[\cite{Schonmann}]
	\label{lem:alpha_m}
	The function $\alpha(p)$ is given by taking a limit of $\alpha_m(p)$, that is,
	\[
	\lim_{m \to \infty} \alpha_m(p)= \alpha(p)
	\]
	for all $p \in [0,1]$.
	\end{lem}
	
	\begin{proof}
	It is clear that for all $m \ge 0, \alpha_m(p) \le \alpha_{m+1}(p)\le \alpha(p)$.
	Then $\{\alpha_m(p)\}_{m\ge0}$ converges and 
	we have $\displaystyle\lim_{m \to \infty} \alpha_m(p) \le \alpha(p)$.
	On the other hand, by definition of $\alpha(p)$, 
	for any small $\epsilon>0$, there is an $n$ such that
	\[
	\alpha(p)-\epsilon \le \mathbb{P}_p(o \leftrightarrow (n,0))^{\frac{1}{n}}.
	\]
	By definition of $\alpha_m(p)$, 
	for any $n \ge 1$, we have
	\[
	\mathbb{P}_p^{G_{\bullet,m}}(o \leftrightarrow (n,0))
	\le \alpha_m(p)^n.
	\]
	From these two inequalities, we have
	\[
	\left( \alpha(p)-\epsilon \right)^n
	\le \mathbb{P}_p(o \leftrightarrow (n,0))
	=\lim_{m \to \infty}\mathbb{P}_p^{G_{\bullet,m}}(o \leftrightarrow (n,0))
	\le \lim_{m \to \infty}\alpha_m(p)^n.
	\]
	Hence we have 
	$\displaystyle\alpha(p)-\epsilon \le  \lim_{m \to \infty}\alpha_m(p)$.	
	It completes the proof.
	\end{proof}

	Let $\pi$ be a natural projection from $T_d \Box \mathbb{Z}$ to $T_d$.
	We define functions $\alpha^\prime(p), \alpha_m^\prime(p)$ 
	similarly to $\alpha(p), \alpha_m(p)$.
	\begin{align*}
	\alpha^\prime(p)
	&=\sup_{n \ge 1} \mathbb{P}_p(o \leftrightarrow \pi^{-1}(n))^{\frac{1}{n}},\\
	\alpha_m^\prime(p)
	&=\lim_{n \to \infty} \mathbb{P}_p^{G_{\bullet,m}}
	(o \leftrightarrow \pi^{-1}(n))^{\frac{1}{n}}
	=\sup_{n \ge 1} \mathbb{P}_p^{G_{\bullet,m}}
	(o \leftrightarrow \pi^{-1}(n))^{\frac{1}{n}}.
	\end{align*}
	We check on the existence of a limit defining the function $\alpha_m(p)$.
	Let $E_n$ be an event that all edges of $\pi^{-1}(n) \cap G_{\bullet,m}$ are open.
	By using FKG inequality, we have
	\begin{align*}
	\mathbb{P}_p^{G_{\bullet,m}}(o \leftrightarrow \pi^{-1}(n+l))
	&\ge \mathbb{P}_p^{G_{\bullet,m}}(o \leftrightarrow \pi^{-1}(n+l) \cap E_n)\\
	&=  \mathbb{P}_p^{G_{\bullet,m}}
	(o \leftrightarrow \pi^{-1}(n) \cap E_n \cap ((n,0) \leftrightarrow \pi^{-1}(n+l)))\\
	&\ge p^{2m} \mathbb{P}_p^{G_{\bullet,m}}(o \leftrightarrow \pi^{-1}(n))
	\mathbb{P}_p^{G_{\bullet,m}}(o \leftrightarrow \pi^{-1}(l))
	\end{align*}
	for all $n,l \ge 0$.
	By using Fekete's subadditive lemma, the existence of the limit is ensured,
	and we have
	\[
	\alpha_m^\prime(p)
	=\sup_{n \ge 1} \mathbb{P}_p^{G_{\bullet,m}}
	(o \leftrightarrow \pi^{-1}(n))^{\frac{1}{n}}.
	\]
	Similarly to Lemma \ref{lem:alpha_m},
	we can show $\displaystyle\lim_{m \to \infty} \alpha_m^\prime(p)
	= \alpha^\prime(p)$.

	\begin{lem}
	\label{lem:alpha^prime}
	For all $p \in [0,1]$, we have $\alpha(p)=\alpha^\prime(p)$.
	\end{lem}
	
	\begin{proof}
	It is clear that $\alpha_m(p) \le \alpha_m^\prime(p)$.
	On the other hand, we have
	\begin{align*}
	\mathbb{P}_p ^{G_{\bullet,m}}(o \leftrightarrow \pi^{-1}(n))
	&\le \sum_{|k| \le m} \mathbb{P}_p^{G_{\bullet,m}}(o \leftrightarrow (n,k))
	\le (2m+1) \alpha_m(p)^n,\\
	\alpha_m(p)^\prime=
	\lim_{n \to \infty } \mathbb{P}_p^{G_{\bullet,m}}
	(o \leftrightarrow \pi^{-1}(n))^{\frac{1}{n}}
	&\le \lim_{n \to \infty} (2m+1)^{\frac{1}{n}} \alpha_m(p)
	=\alpha_m(p).
	\end{align*}
	Then $\alpha_m(p)=\alpha_m(p)^\prime$ holds for all $m \ge 0$.
	By taking the limit, we have $\alpha(p)=\alpha^\prime(p)$.
	\end{proof}

	\noindent \textit{Proof of Lemma \ref{lem:alpha(p>p_c)}}
	We know another definition of the critical probability.
	Let $(o \leftrightarrow \infty)$ be an event that 
	there exists an infinite open path from the origin.
	Then we have
	\[
	p_c= \sup \left\{ p\in[0,1] \mid \mathbb{P}_p(o \leftrightarrow \infty)=0 \right\}.
	\]
	Let $B(n) \subset T_d$ be a $n$-ball whose center is the origin,
	and we set $G_{n,\bullet}=B(n) \Box \mathbb{Z}$.
	If $(o \leftrightarrow \infty)$ occurs on $T_d \Box \mathbb{Z}$,
	then $(o \leftrightarrow \partial B(n) \Box \mathbb{Z})$ or 
	$(o \leftrightarrow \infty)$ occur on $G_{n,\bullet}$.
	It is clear that $p_c(G_{n,\bullet})=1$. 
	Hence, by using Lemma \ref{lem:alpha^prime}, we have
	\begin{align*}
	\mathbb{P}_p(o \leftrightarrow \infty)
	&\le \mathbb{P}_p(o \leftrightarrow \partial B(n) \Box \mathbb{Z})
	+\mathbb{P}_p^{G_{n,\bullet}}(o \leftrightarrow \infty)\\
	&\le \sum_{x \in \partial B(n)} \mathbb{P}_p(o \leftrightarrow \pi^{-1}(x))
	\le d(d-1)^{n-1} \alpha(p)^n
	\end{align*}
	for all $p<1$.
	The right-hand side goes to 0 if $\alpha(p)<1/(d-1)$.
	Since $\mathbb{P}_p(o \leftrightarrow \infty)>0$ holds when $p>p_c$,
	Then we have $\alpha(p) \ge 1/(d-1)$ for all $p>p_c$.
	\begin{flushright} $\Box$ \end{flushright}

	\section{Extension of some theorems}
	\label{sc:extension}
	In Bernoulli percolation,
	some theorems can only be applied to an event which depends on finite edges.
	For edge subset $F$, let $[ \omega]_F$ be a subset of $\Omega$
	whose elements have the same configuration as $\omega$ on $F$.
	An event $A$ is said to depend on (only) finite edges 
	if there exists finite edge set $F$ such that 
	$[\omega]_F \subset A$ or $[\omega]_F \cap A=\emptyset$ holds 
	for all $\omega \in \Omega$.
	For $\omega, \tau \in \Omega$, we write $\omega \le \tau$ 
	if $\omega(e) \le \tau(e)$ holds for all $e \in E$.
	An event $A$ is called increasing
	if $\tau \in A$ whenever $\omega \in A$ and $\omega \le \tau$.

	\begin{thm}[\cite{Grimmett} (2.39)]
	\label{thm:stri}
	Let $A$ be an increasing event which depends on finite edges.
	Then we have
	\[
	\mathbb{P}_{p^\gamma}(A) \le \mathbb{P}_p(A)^\gamma
	\]
	for all $0<p<1$ and $\gamma \ge 1$.
	\end{thm}

	For two events $A$ and $B$,
	$A \circ B$ is defined as the event that A and B occur on disjoint edge sets,
	formulated by 
	\[
	A \circ B = \left\{ \omega \in \Omega \mid
	\exists \text{finite disjoint } K,L \subset E \text{ s.t. } 
	[\omega]_K \subset A, [\omega]_L \subset B \right\}.
	\]
	
	\begin{thm}[BK inequality \cite{BK}]
	\label{thm:BK}
	Let $A,B$ be increasing events which depends on finite edges.
	Then we have 
	\[
	\mathbb{P}_p(A \circ B) \le \mathbb{P}_p(A) \mathbb{P}_p(B).
	\]
	\end{thm}
	
	We will extend these two theorems so that 
	it can be applied to certain events which depends on infinite edges.
	Let $K\subset E$ be a finite edge subset,
	and $L \subset E$ be an edge subset which may be infinite.
	An event $(K \leftrightarrow L)$ is called a connection event,
	for example $(o \leftrightarrow x)$ or $(o \leftrightarrow \pi^{-1}(x))$.
	It is clear that a connection event is an increasing event,
	and depends on infinite edges in general.

	\begin{lem}
	\label{lem:stri+}
	Let $A$ be a connection event.
	Then we have
	\[
	\mathbb{P}_{p^\gamma}(A) \le \mathbb{P}_p(A)^\gamma
	\]
	for all $0<p<1$ and $\gamma \ge 1$.
	\end{lem}

	\begin{lem}
	\label{lem:BK+}
	Let $A,B$ be connection events.
	Then we have 
	\[
	\mathbb{P}_p(A \circ B) \le \mathbb{P}_p(A) \mathbb{P}_p(B).
	\]
	\end{lem}
	
	\noindent \textit{Proof of Lemma \ref{lem:stri+} and Lemma \ref{lem:BK+}}
	Let $A=(K \leftrightarrow L)$,
	$\Gamma$ be a set of all paths between $K$ and $L$.
	Then we have 
	\[
	A= \bigcup_{q \in \Gamma} (q : open).
	\]
	Each event $(q:open)$ is increasing and it depends on finite edges.
	For any $\epsilon>0$, 
	there exists a finite subset $\Gamma^\prime \subset \Gamma$ such that 
	\[
	\mathbb{P}_{p^\gamma}(A)- \epsilon \le \mathbb{P}_{p^\gamma}
	\left( \bigcup_{q \in \Gamma^\prime} (q : open) \right).
	\]
	The event in the right-hand side is increasing and it depends on finite edges.
	Then by using Theorem \ref{thm:stri}, we have
	\[
	\mathbb{P}_{p^\gamma}(A)- \epsilon
	\le \mathbb{P}_p\left( \bigcup_{q \in \Gamma^\prime} (q : open) \right)^\gamma
	\le \mathbb{P}_p(A)^\gamma.
	\]
	It completes the proof of Lemma \ref{lem:stri+}.
	Next we will show Lemma \ref{lem:BK+}.
	Let $A_i=(K_i \leftrightarrow L_i)$,
	$\Gamma_i^{(n)}$ be a set of all paths between 
	$K_i$ and $L_i$ with length $n$ or less,
	$\displaystyle C_i^{(n)}=\bigcup_{q \in \Gamma_i^{(n)} }(q:open)$ for $i=1,2$.
	Then we have
	\[
	A_i=\bigcup_{n \ge 1} C_i^{(n)}.
	\]
	If $\omega \in A_1 \circ A_2$,
	then there exists finite disjoint subset $F_1,F_2 \subset E$ such that
	$[\omega]_{F_i} \subset A_i$.
	We take $n=\max \{ |F_i| \}$, then we have $[\omega]_{F_i} \subset C_i^{(n)}$,
	that is $\omega \in C_1^{(n)} \circ C_2^{(n)}$.
	Hence, we have 
	\[
	A_i \circ A_2 \subset \bigcup_{n \ge 1} 
	\left(  C_1^{(n)} \circ C_2^{(n)} \right).
	\]
	Since $K_1, K_2$ are finite,
	each of the events $ C_1^{(n)}$ and $C_2^{(n)}$ is increasing 
	and depends on finite edges.
	Then by using Theorem \ref{thm:BK}, we have
	\[
	\mathbb{P}_p(A_1 \circ A_2)
	\le \lim_{n \to \infty} \mathbb{P}_p \left(C_1^{(n)} \circ C_2^{(n)}\right)
	\le \lim_{n \to \infty} \left( \mathbb{P}_p\left(C_1^{(n)}\right) 
	\mathbb{P}_p\left(C_2^{(n)}\right) \right)
	= \mathbb{P}_p(A_1)\mathbb{P}_p(A_1).
	\]
	\begin{flushright} $\Box$ \end{flushright}

	At the end of this section,
	we show the first half of Theorem \ref{thm:conti}. 
	By using Lemma \ref{lem:stri+},
	we have 
	\[
	\alpha(p^\gamma) 
	= \lim_{n \to \infty} \mathbb{P}_{p^\gamma} (o \leftrightarrow (n,0))^{\frac{1}{n}}
	\le  \lim_{n \to \infty} \mathbb{P}_p (o \leftrightarrow (n,0))^{\frac{\gamma}{n}}
	=\alpha(p)^\gamma.
	\]
	We know $\alpha(p)<1$ for all $p \le p_u$ from Theorem \ref{thm:Schonmann}.
	Then we have 
	\[
	\alpha(p^\gamma) \le \alpha(p)^\gamma <\alpha(p)
	\]
	for all $p \le p_u$ and $\gamma>1$.
	Therefore $\alpha(p)$ is a strictly increasing function on $[0,p_u]$.

	\section{Connection event}
	\label{sc:conn.event}
	In Section \ref{sc:extension}, we define a connection event,
	and we prepared some lemmas concerning connection events.
	In this section, we prepar one more lemma concerning connection events.
	A graph $G$ is called nonamenable if the Cheeger constant of $G, h(G)$,
	defined by
	\[	
	h(G)= \inf \left\{ \frac{|\partial S|}{|S|} \mid S \subset V, |S|< \infty \right\}.
	\] 
	 is positive.
	\begin{thm}[\cite{BLPS}]
	\label{thm:Cheeger}
	Let $G$ be a nonamenable Cayley graph. Then we have
	\[	
	\mathbb{P}_{p_c}(o \leftrightarrow \infty)=0.
	\]
	\end{thm}
	It is well-known that $h(T_d \Box \mathbb{Z})=d-2$,
	that is, $T_d \Box \mathbb{Z}$ is a nonamenable graph for all $d\ge3$.
	Also, let $S=\{a_1, \ldots, a_d, b\}$ be a generating set,
	and $ \Gamma=<a_i,b| a_i^{-1}=a_i,a_ib=ba_i>$ be a group generated by $S$,
	then $T_d \Box \mathbb{Z}$ is a Cayley graph of $(\Gamma,S)$.
	Therefore, we can use this theorem for $T_d \Box \mathbb{Z}$.

	\begin{lem}
	\label{lem:conn.conti}
	Let $G=T_d \Box \mathbb{Z}$,
	and $A$ be a connection event. 
	Then $\mathbb{P}_p(A)$ is continuous on $[0,p_c]$.
	\end{lem}
	
	\begin{proof}
	It is clear that $\mathbb{P}_p(A)$ is left-continuous similar to $\alpha(p)$, since
	\[
	\mathbb{P}_p(A)= \sup_{k \ge1} \mathbb{P}_p^{B(k)} (A)
	\]
	where $B(k)$ is a $k$-ball.
	We will prove that 
	$\mathbb{P}_p(A)$ is right-continuous on $[0,p_c]$ in this section.
	We prepare another definition of $\mathbb{P}_p$ (ref: \cite{Grimmett} section1.3).
	Let $\Omega^\prime=[0,1]^E$, 
	$\mu_e$ be a uniform distribution on $[0,1]$ for each $e \in E$, 
	and $\mu=\prod_{e\in E} \mu_e$ be a probability measure on $\Omega^\prime$.
	For any $p \in [0,1]$ and $\{ X_e \}_{e \in E} \in \Omega^\prime$, 
	let $\omega_p$ be a configuration defined by
	\[
	\omega_p(e)=\mathbb{1}_{\{X_e <p\} }
	\]
	for any $e \in E$.
	We define a map $f_p$ from $\Omega^\prime$ to $\Omega=\{0,1\}^E$,
	by $f_p(\{X_e\})=\omega_p$.
	Then the pushforward measure of $\mu$ is the same as $\mathbb{P}_p$,
	that is,
	\[
	{f_p}_{\ast}(\mu)=\mathbb{P}_p.
	\]
	By using this equation, we have
	\[
	\mathbb{P}(A)=\mu(\omega_p \in A).
	\]
	Let $p_0 \in [0,p_c]$ be a fixed point.
	For any $p>p_0$ and , we have
	\[
	\mathbb{P}_p(A)-\mathbb{P}_{p_0}(A)
	=\mu( \omega_p \in A, \omega_{p_0} \not\in A).
	\]  
	Hence, by taking the limit, we have
	\[
	\lim_{p \downarrow p_0} \left( \mathbb{P}_p(A)-\mathbb{P}_{p_0}(A) \right)
	=\mu(\forall p>p_0, \omega_p \in A, \omega_{p_0} \not\in A).
	\]
	Let $A=(K\leftrightarrow L)$,
	we consider $(\omega_{p_0} \not\in A)$ occures.
	By Theorem \ref{thm:Cheeger},
	there exists no infinite path from $x$ on $\omega_{p_0}$ almost surely 
	for any $x \in K$ and any $p_0 \in [0,p_c]$.
	Hence, connected components containing elements in $K$ are finite.
	Let $H$ be a finite subgraph which contains all of these connected components.
	If $\omega_p \in A$ holds,
	then there exists at least one edge $e$ on $H$ such that 
	$\omega_{p_0}(e)=0$ and $\omega_p(e)=1$.
	If $\omega_p \in A$ holds for all $p>p_0$,
	then there exists at least one edge $e$ on $H$ such that $X_e = p_0$.
	Hence, we have
	\[
	\lim_{p \downarrow p_0} \left( \mathbb{P}_p(A)-\mathbb{P}_{p_0}(A) \right)
	\le \sum_{e \in E(H)} \mu(X_e=p_0) \le 0.
	\]
	It ends the proof.
	\end{proof}
	\section{Another function $\beta(p)$}
	\label{sc:beta}
	It is left to prove the second half of Theorem \ref{thm:conti}.
	We prepare another function $\beta(p)$ to prove it, defined by
	\[
	\beta(p)
	= \lim_{m \to \infty} \mathbb{P}_p(o \leftrightarrow (0,m))^{\frac{1}{m}}
	= \sup_{m \ge 1} \mathbb{P}_p(o \leftrightarrow (0,m))^{\frac{1}{m}}.
	\]
	The existence of the limit is ensured and 
	it can be written as a supremum similar to $\alpha(p)$.
	By using FKG inequality and the homogeneity of $T_d \Box \mathbb{Z}$, we have 
	\[
	\mathbb{P}_p(o \leftrightarrow (2n,0)) \ge 
	\mathbb{P}_p\left((o\leftrightarrow(n,m)) \cap ((n,m)\leftrightarrow (2n,0))\right)
	\ge \mathbb{P}_p(o\leftrightarrow(n,m))^2.
	\]
	Hence, we have
	\[
	 \mathbb{P}_p(o\leftrightarrow(n,m))
	\le \mathbb{P}_p(o \leftrightarrow (2n,0))^{\frac{1}{2}}
	\le \alpha(p)^n
	\]
	for each $(n,m)$.
	Similarly, we have 
	\[
	\mathbb{P}_p(o\leftrightarrow(n,m)) \le \beta(p)^m.
	\]
	For each $n \ge1$, we define $I_n(p)$ by
	\[
	I_n(p)=\sum_{k \in \mathbb{Z}} \mathbb{P}_p(o \leftrightarrow (n,k)).
	\]
	Since $\mathbb{P}_p(o\leftrightarrow(n,m)) \le \beta(p)^m$,
	it is well-defined when $\beta(p)<1$.

	\begin{lem}
	\label{lem:beta<1}
	For any $p <p_u$, we have
	\[
	\beta(p)<1.
	\]
	\end{lem}
	This lemma is shown in the next section.
	We assume Lemma \ref{lem:beta<1} holds,
	and only consider when $p<p_u$.
	
	\begin{lem}
	\label{lem:I_n}
	For any $n,l \ge 1$, we have
	\[
	I_{n+l}(p) \le I_n(p) I_l(p).
	\]
	\end{lem}
	
	\begin{proof}
	By using Lemma \ref{lem:BK+}, we have
	\begin{align*}
	I_{n+l}(p)
	&=\sum_{k \in \mathbb{Z}} \mathbb{P}_p(o \leftrightarrow (n+l,k))
	= \sum_{k \in \mathbb{Z}} \mathbb{P}_p \left( \bigcup_{t \in \mathbb{Z}}
	(o \leftrightarrow (n,t)) \circ ((n,t) \leftrightarrow (n+l,k)) \right)\\
	&\le \sum_{k \in \mathbb{Z}} \sum_{t \in \mathbb{Z}} 
	\mathbb{P}_p(o \leftrightarrow (n,t)) \mathbb{P}_p(o \leftrightarrow (l,k-t))
	\le I_n(p) I_l(p).
	\end{align*}
	\end{proof}

	Therefore, we define a function $\eta(p)$ by
	\[
	\eta(p)=\lim_{n \to \infty} I_n(p)^{\frac{1}{n}}
	= \inf_{n \ge 1} I_n(p)^{\frac{1}{n}}.
	\]
	
	\begin{lem}
	The function $\eta(p)$ is right-continuous on $[0,p_c]$.
	\end{lem}
	
	\begin{proof}
	If $I_n(p)$ is a continuous on $[0,p_c]$ for any $n \ge1$,
	then $\eta(p)$ is the infimum of a right-continuous of $p$.
	Therefore, $\eta(p)$ is a right-continuous on $[0,p_c]$.
	We define $I_n^{(m)}(p)$ by
	\[
	I_n^{(m)}(p)=\sum_{|k| \le m} \mathbb{P}_p(o \leftrightarrow (n,k)).
	\]
	By Lemma \ref{lem:conn.conti},
	$I_n^{(m)}(p)$ is continuous on $[0,p_c]$,
	and we have
	\begin{align*}
	I_n(p)-I_n^{(m)}(p)
	&= \sum_{|k|>m} \mathbb{P}_p(o \leftrightarrow (n,k))
	\le 2 \sum_{k>m} \beta(p)^k\\
	&=2 \frac{\beta(p)^{m+1}}{1-\beta(p)}
	\le 2 \frac{\beta(p_c)^{m+1}}{1-\beta(p_c)}
	\end{align*}
	for all $p \in [0,p_c]$. Therefore, 
	$\{I_n^{(m)}(p)\}_{m \ge 0} $ uniformly converges to $I_n(p)$ on $[0,p_c]$.
	Hence $I_n(p)$ is continuous on $[0,p_c]$.
	\end{proof}
	We will prove the second half of Theorem \ref{thm:conti}.
	Now we know that $\alpha(p)$ is left-continuous on $[0,1]$
	and $\eta(p)$ is right-continuous on $[0,p_c]$.
	\begin{lem}
	For any $p \in [0,p_u)$, we have $\alpha(p)=\eta(p)$. 
	In particular, $\alpha(p)$ is continuous on $[0,p_c]$.
	\end{lem}
	
	\begin{proof}
	We define $\eta_m(p)$ by
	\[	
	\eta_m(p)= \liminf_{n \to \infty} I_n^{(m)}(p)^{\frac{1}{n}}.
	\]
	For all $m \ge 0, n \ge 1, p \in [0, p_u)$, it is clear taht
	$I_n^{(m)}(p)^{1/n} \le I_n(p)^{1/n}$ holds.
	Then we have $\eta_m(p) \le \eta(p)$.
	Hence we have $\displaystyle \lim_{m \to \infty} \eta_m(p) \le \eta(p)$.
	First, we show that $\displaystyle \lim_{m \to \infty} \eta_m(p) = \eta(p)$.
	By the definition of $\eta_m(p)$,
	for any $\epsilon>0, m \ge 0$, there exists an $n \ge 1$ such that
	\[
	I_n^{(m)}(p)^{\frac{1}{n}}-\epsilon \le \eta_m(p).
	\] 
	By the definition of $\eta(p)$, for all $n\ge 1$, we have 
	\begin{align*}
	\eta(p)^n
	& \le I_n(p)
	= \lim_{m \to \infty} \sum_{|k| \le m} \mathbb{P}_p(o \leftrightarrow (n,k))
	= \lim_{m \to \infty} \left( I_n^{(m)}(p)^{\frac{1}{n}} \right)^n,\\
	\eta(p) &\le \lim_{m \to \infty} I_n^{(m)}(p)^{\frac{1}{n}}.
	\end{align*}
	Therefore we have 
	\[
	\eta(p)-\epsilon
	\le \lim_{m \to \infty} \eta_m(p)
	\]
	for any $\epsilon>0$.
	it completes the proof of 
	$\displaystyle \lim_{m \to \infty} \eta_m(p) = \eta(p)$.
	Next, for all $n \ge 1$, it is clear that 
	$\mathbb{P}_p(o \leftrightarrow (n,0))^{1/n} \le I_n(p)^{1/n}$.
	Then we have $\alpha(p) \le \eta(p)$.
	For any $\epsilon>0$, there exists $m$ such that 
	\[
	\eta(p)-\frac{\epsilon}{2} \le \eta_m(p).
	\]
	By the definition of $\eta_m(p)$,
	there exisits $N \ge 1$ such that 
	\[
	\eta_m(p) - \frac{\epsilon}{2} \le I_n^{(m)}(p)^{\frac{1}{n}}
	\]
	for all $n \ge N$.
	By the inequality $\mathbb{P}_p( o\leftrightarrow (n,k)) \le \alpha(p)^n$,
	we have
	\[
	I_n^{(m)}(p) \le (2m+1) \alpha(p)^n.
	\]
	Therefore, from three above inequalty, we have
	\[	
	\eta(p)-\epsilon \le (2m+1)^{\frac{1}{n}} \alpha(p)
	\]
	for any $\epsilon>0$.
	The right-hand side goes to $\alpha(p)$ as $n \to \infty$.
	It completes the proof.
	\end{proof}
	Then we showed Theorem \ref{thm:conti} 
	if Lemma \ref{lem:beta<1} holds.

	\section{Proof of Lemma \ref{lem:beta<1}}
	\label{sc:pf.beta<1}
	In this section, we prove Lemma \ref{lem:beta<1}.
	Our method is based on \cite{Liggett},
	this paper is about contact process,
	we aplly it to percolation process.
	\begin{lem}
	\label{lem:beta.inf}
	For any $p \in [0, p_u)$, we have
	\[	
	\inf_{m \ge 0} \mathbb{P}_p(o \leftrightarrow (0,m))=0.
	\]
	\end{lem}
	
	\begin{proof}
	We recall $G_{n, \bullet}$ is a subgraph defined by 
	$G_{n, \bullet}=B(n) \Box \mathbb{Z}$
	where $B(n)$ is an $n$-ball whose center is the origin.
	Since $p_c(G_{n, \bullet})=1$,
	we have
	\[
	\inf_{m \ge 0} \mathbb{P}_p^{G_{n, \bullet}}(o \leftrightarrow (0,m))=0
	\]
	for any $p<1$.
	If $(o \leftrightarrow (0,m))$ occurs on $G$, 
	then at least one of events $(o \leftrightarrow (0,m))$ on $G_{n, \bullet}$ or 
	there exists an open path between $o$ and $(0,m)$ 
	which is not contained $G_{n, \bullet}$ occurs.
	If the latter occurs,
	there exists $x \in \partial B(n)$ such that 
	$(o \leftrightarrow \pi^{-1}(x))$ and $(o \leftrightarrow \pi^{-1}(x))$ 
	occur on disjoint edge subsets.
	Then we have
	\[
	\mathbb{P}_p(o \leftrightarrow (0,m))
	\le \mathbb{P}_p^{G_{n, \bullet}}(o \leftrightarrow (0,m))
	+\mathbb{P}_p \left( \bigcup_{x \in \partial B(n)} 
	(o \leftrightarrow \pi^{-1}(x)) \circ ((0,m) \leftrightarrow \pi^{-1}(x)) \right)
	\]
	for all $n \ge 1$.
	By Lemma \ref{lem:alpha^prime} and Lemma \ref{lem:BK+}, we have
	\begin{align*}
	\mathbb{P}_p \left( \bigcup_{x \in \partial B(n)} 
	(o \leftrightarrow \pi^{-1}(x)) \circ (o \leftrightarrow \pi^{-1}(x)) \right)
	&\le \sum_{x \in \partial B(n)} \mathbb{P}_p(o \leftrightarrow \pi^{-1}(x)) 
	\mathbb{P}_p((0,m) \leftrightarrow \pi^{-1}(x))\\
	&\le d(d-1)^{n-1} \alpha(p)^{2n}.
	\end{align*}
	By Theorem \ref{thm:Schonmann} and by the fact that 
	$\alpha(p)$ is strictly increasing on $[0,p_u]$, we have
	\[
	\alpha(p) < \frac{1}{\sqrt{d-1}}
	\]
	for all $p \in [0,p_u)$.
	Therefore, we have
	\[
	\inf_{m \ge 0} \mathbb{P}_p(o \leftrightarrow (0,m))
	\le d(d-1)^{n-1} \alpha(p)^{2n}
	\rightarrow 0
	\]
	as $n \to \infty$.
	It ends the proof.
	\end{proof}
	
	We define the level function $L(x)$ from $T_d$ to $\mathbb{Z}$.
	Let $\gamma$ be an infinite geodesic from the origin on $T_d$.
	First, we define $L(x)=-|x|=-d(o, x)$ when $x \in \gamma$.
	Next, when $x \not\in \gamma$,
	there exists only one vertex $x^\prime \in \gamma$
	such that $d(x, \gamma)=d(x,x^\prime)$.
	Then we define $L(x)=L(x^\prime)+d(x,x^\prime)$.
	This level function $L(x)$ is based on the origin.
	Let $\gamma_y$ be a unique infinite geodesic from $y$ 
	such that $|\gamma \cap \gamma_y|=\infty$.
	We define $L_y (x)$ in the same way by replacing the origin with $y$
	and $\gamma$ with $\gamma_y$.
	
	\begin{lem}
	\label{lem:level}
	For any $x,y \in T_d$, we have
	\[
	L(x)=L(y)+L_y(x).
	\]
	\end{lem}
	
	\begin{proof}
	Let two vertices $x^\prime, y^\prime$ such that 
	$d(x, \gamma)=d(x,x^\prime), d(y, \gamma)=d(y,y^\prime)$.
	Then we have
	\begin{align*}
	L(x)&=-|x^\prime|+d(x,x^\prime),\\
	L(y)&=-|y^\prime|+d(y,y^\prime).
	\end{align*}
	If $x^\prime \in \gamma_y$, we have
	\begin{align*}
	L_y(x)&=-d(y,x^\prime)+d(x,x^\prime)
	=-d(y,y^\prime)+|y^\prime|-|y^\prime|-d(y^\prime, x^\prime)+d(x,x^\prime)\\
	&=-|x^\prime|+d(x,x^\prime)+|y^\prime|-d(y,y^\prime)
	=L(x)-L(y).
	\end{align*}
	If $x^\prime \not\in \gamma_y$, we have
	\begin{align*}
	L_y(x)&=-d(y,y^\prime)+d(y^\prime,x)
	=-d(y,y^\prime)+|y^\prime|-|y^\prime|+d(y^\prime, x^\prime)+d(x,x^\prime)\\
	&=-|x^\prime|+d(x,x^\prime)+|y^\prime|-d(y,y^\prime)
	=L(x)-L(y).
	\end{align*}
	It ends the proof.
	\end{proof}

	For $n \ge 0, z \in \mathbb{R}_{>0}$, we define $a_n(z)$ by
	\[
	a_n(z)=\sum_{|x|=n}z^{L(x)}.
	\]
	Stacey \cite{Stacey} has computed the number of vertices $x \in T_d$
	satisfying $|x|=n, L(x)=n-2t$:
	\[
	\begin{cases}
	b^n & (t=0)\\
	(b-1)b^{n-t-1} &(1\le t \le n-1)\\
	1 &(t=n)
	\end{cases}
	\]
	where $b=d-1$.
	By using this formula, Ligeett \cite{Liggett} showed the following equations.
	\begin{align*}
	a_n(z)
	&=(bz)^n +\sum_{t=1}^{n-1} (b-1)b^{n-t-1}z^{n-2t}+z^{-n}\\
	&=\begin{cases}
	\frac{b^{n-1}z^n (b^2 z^2 -1 )+ z^{-n} (z^2-1)}{bz^2-1} & (bz^2\not= -1)\\
	\sqrt{b}^n((n+1)-b^{-1}(n-1)) & (bz^2=1),
	\end{cases}\\
	a_n(1/bz)&=a_n(z).
	\end{align*}
	For $m\ge 0$, we define $J_m(p,z)$ by
	\[
	J_m(p,z)
	=\sum_{x \in T_d} \mathbb{P}_p(o \leftrightarrow (x,m))z^{L(x)}
	=\sum_{n \ge0} \mathbb{P}_p(o \leftrightarrow (n,m)) a_n(z).
	\]
	Now we consider $p<p_u$, that is, $\alpha(p)<1/\sqrt{b}$.
	Then we have $\alpha(p)<1/(\alpha(p)b)$.
	We know that $\mathbb{P}_p(o \leftrightarrow (n,m)) \le \alpha(p)^n$ holds.
	Then there exists a constant $C(p,z)$ which does not depend on $m$ such that
	\[
	J_m(p,z) \le C(p,z)
	\] for all $z \in (\alpha(p), 1/\alpha(p)b)$.
	Therefore, $J_m(p,z)$ is well-defined for $z \in (\alpha(p), 1/(\alpha(p)b)$.
	
	\begin{lem}
	For any $m,l \ge0$, we have $J_{m+l}(p,z) \le J_m(p,z)J_l(p,z)$.
	\end{lem}
	\begin{proof}
	By Lemma \ref{lem:BK+} and Lemma \ref{lem:level}, we have
	\begin{align*}
	J_{m+l}(p,z)
	&=\sum_{x \in T_d} \mathbb{P}_p \left( 
	\bigcup_{y \in T_d}(o \leftrightarrow (y,m)) \circ ((y,m) \leftrightarrow (x,m+l))
	\right)z^{L(x)}\\
	& \le \sum_{y \in T_d} \mathbb{P}_p (o \leftrightarrow (y,m)) z^{L(y)}
	\sum_{x \in T_d} \mathbb{P}_p ((y,m) \leftrightarrow (x,m+l)) z^{L_y(x)}\\
	&=J_m(p,z)J_l(p,z).
	\end{align*}
	\end{proof}
	From this lemma, we can define $\phi(p,z)$ by
	\[
	\phi(p,z) = \lim_{m \to \infty} J_m(p,z)^\frac{1}{m}
	=\inf_{ m\ge 0}  J_m(p,z)^\frac{1}{m}.
	\]
	By definition of $\phi(p,z)$, we have
	\[
	\phi(p,z)^m \le J_m(p,z).
	\]
	Since $L(o)=0$, 
	we have $\mathbb{P}_p(o \leftrightarrow (0,m)) \le J_m(p,z).$
	Then $\beta(p) \le \phi(p,z)$ holds.
	Therefore, if there exists $z$ such that 
	$\displaystyle \inf_{m \ge 0} J_m(p,z)<1$,
	then $\phi(p,z)<1$ holds.
	It leads that $\beta(p)<1$.
	Next lemma completes the proof of Lemma \ref{lem:beta<1}.

	\begin{lem}
	For any $z \in (\alpha(p), 1/\alpha(p)b)$, we have
	\[
	\inf_{m \ge 0}J_m(p,z)=0.
	\]
	\end{lem}
	\begin{proof}
	Since $a_n(1/bz)=a_n(z)$, 	we have $J_m(p,1/bz)=J_m(p,z)$.	
	Then we only consider $z \in [1/\sqrt{b}, 1/\alpha(p)b)$.
	For any $z_0 \in (z, 1/\alpha(p)b)$,
	we have
	\begin{align*}
	\frac{a_n(z)}{a_n(z_0)}
	&= \frac{bz_0^2-1}{bz^2-1} \cdot \frac{b^{n-1}z^n(b^2z^2-1)+ z^{-n} (z^2-1)}
	{b^{n-1}z_0^n(b^2z_0^2-1)+ z_0^{-n} (z_0^2-1)}\\
	&=\frac{bz_0^2-1}{bz^2-1} \cdot \frac{b^{-1}(b^2z^2-1)+ (1/bz^2)^n (z^2-1)}
	{b^{-1}(z_0/z)^n(b^2z_0^2-1)+ (1/bzz_0)^n (z_0^2-1)}\\
	&\to 0 \text{\qquad as $n \to \infty$}.
	\end{align*}
	Therefore, for any $z_0\in (z, 1/\alpha(p)b), \epsilon>0$, 
	there exists $N \in \mathbb{N}$ such that 
	\[
	\frac{a_n(z)}{a_n(z_0)}\le \frac{\epsilon}{C(p,z_0)}
	\]
	for all $n \ge N$, where $C(p,z_0)$ is a constant which does not depend on $m$
	such that $J_m(p.z_0) \le C(p,z_0)$.
	Then we have
	\begin{align*}
	J_m(p,z)
	&= \sum_{n \ge 0} a_n(z) \mathbb{P}_p(o \leftrightarrow (n,m))\\
	&= \sum_{n \ge N} a_n(z_0) \cdot \frac{a_n(z)}{a_n(z_0)} 
	\mathbb{P}_p(o \leftrightarrow (n,m))
	+ \sum_{n <N} a_n(z) \mathbb{P}_p(o \leftrightarrow (n,m))\\
	&\le \frac{\epsilon}{C(p,z_0)} J_m(p, z_0)
	+\sum_{n <N} a_n(z) \mathbb{P}_p(o \leftrightarrow (n,m)).
	\end{align*}
	From Lemma \ref{lem:beta.inf} we have
	\[
	\inf_{m \ge 0} J_m(p,z) \le \epsilon
	+ \sum_{n<N}a_n(z) \inf_{m \ge 0} \mathbb{P}_p(o \leftrightarrow (n,m))
	= \epsilon.
	\]
	\end{proof}

{}
\noindent
{Mathematical Institute \\
 Tohoku University \\
 Sendai 980-8578 \\
 Japan\\
E-mail:kohei.yamamoto.t1@dc.tohoku.ac.jp}


\begin{thebibliography}{}
\bibitem{Aizenman2}Aizenman, M.; Barsky, D. J. 
	Sharpness of the phase transition in percolation models.
	\textit{Comm. Math. Phys.} \textbf{108, no.3}, 489--526. (1987).
\bibitem{Aizenman}Aizenman, M.; Kesten, H.; Newman, C. M.
	Uniqueness of the infinite cluster and continuity of connectivity functions
	for short and long range percolation.
	\textit{Comm. Math. Phys.} \textbf{111, no. 4}, 505--531. (1987).
\bibitem{Antunovic}Antunovi\'c, T.; Veseli\'c, I. 
	Sharpness of the phase transition and exponential decay of the subcritical
	cluster size for percolation on quasi-transitive graphs.
	\textit{J. Stat. Phys}. \textbf{130, no. 5}, 983--1009. (2008).
\bibitem{BLPS}Benjamini, I.; Lyons, R.; Peres, Y.; Schramm, O.
	Critical percolation on any nonamenable group has no infinite clusters.
	 \textit{Ann. Probab.} \textbf{27, no. 3}, 1347--1356.(1999).
\bibitem{BK}van den Berg, J.; Kesten, H. 
	Inequalities with applications to percolation and reliability.
	\textit{J. Appl. Probab.} \textbf{22}, 556--569. (1985).
\bibitem{Grimmett}Grimmett, G. R.
	\textit{Percolation.}
	volume 321 of Grundlehren der Mathematischen Wissenschaften.
	Springer-Verlag, Berlin, second edition. (1999).
\bibitem{Grimmett2}Grimmett, G. R.; Newman, C. M.
	Percolation in $\infty+1$ dimensions.
	In \textit{disorder in physical systems.} 167--190. 
	Oxford Univ. Press, New York. (1990).
\bibitem{Hutchcroft1}Hutchcroft, T.
	Critical percolation on any quasi-transitive graph of exponential
	growth has no infinite clusters.
	\textit{C. R. Math. Acad. Sci. Paris.} \textbf{354, no. 9}, 944--947. (2016).
\bibitem{Hutchcroft2}Hutchcroft, T.
	Non-uniqueness and mean-field criticality for percolation on nonunimodular
	transitive graphs.
	\textit{arXiv preprint arXiv:1711.02590.} (2017).
\bibitem{Kesten}
	Kesten, H.
	The critical probability of bond percolation on the square lattice equals 1/2.
	\textit{Comm. Math. Phys.} \textbf{74, no. 1}, 41--59. (1980).
\bibitem{Liggett}Liggett, T. M.
	Branching random walks and contact processes on homogeneous trees.
	\textit{Probab. Theory Related Fields.} \textbf{106, no. 4}, 495--519. (1996).
\bibitem{Schonmann}Schonmann, R. H.
	Percolation in $\infty+1$ dimensions at the uniqueness threshold.
	\textit{Perplexing problems in probability. Progress in Probability}. 
	\textbf{44}, 53--67. Birkhuser, Boston. (1999).
\bibitem{Stacey}Stacey, A. M.
	The existence of an intermediate phase for the contact process on trees. 
	\textit{Ann. Probab}. \textbf{24, no. 4}, 1711--1726. (1996). 
\end{thebibliography}
\end{document}